\documentclass{amsart}
\usepackage[utf8]{inputenc}
\usepackage{amsmath}
\usepackage{amsthm}
\usepackage{amsfonts, dsfont}
\usepackage{amssymb}
\usepackage[all]{xy}
\usepackage{indentfirst}
\usepackage{hyperref}

\usepackage{cleveref}
\usepackage[alphabetic, initials]{amsrefs}

\usepackage{faktor}
\usepackage{xfrac}

\newtheorem{thm}{Theorem}[section]

\newtheorem{theorem}[thm]{Theorem}
\newtheorem{corollary}[thm]{Corollary}
\newtheorem{proposition}[thm]{Proposition}

\newtheorem{lemma}[thm]{Lemma}
\newtheorem*{theorem*}{Theorem}
\newtheorem*{corollary*}{Corollary}

\theoremstyle{definition}

\newtheorem*{defn*}{Definiton}
\newtheorem{example}[thm]{Example}

\newtheorem{remark}[thm]{Remark}

\newcommand{\N}{\mathbb{N}} %% Naturals
\newcommand{\Z}{\mathbb{Z}} %% Integers
\newcommand{\Q}{\mathbb{Q}} %% Rationals
\newcommand{\R}{\mathbb{R}}

\usepackage{verbatim} %% Enables 'comment'
\usepackage{color}
\newcommand{\G}{\Gamma}

\DeclareMathOperator{\Ima}{Im}
\DeclareMathOperator{\dom}{dom}

\DeclareMathOperator{\co}{Comm}
\newcommand{\La}{\Lambda}

\newcommand{\Id}{\mathrm{Id}}

\newcommand{\Sub}{\operatorname{Sub}}

\newcommand{\act}{\!\curvearrowright\!}

\newcommand{\Fix}{\operatorname{Fix}}

\newcommand{\Prob}{\operatorname{Prob}}

\newcommand{\cC}{\mathcal{C}}

\newcommand{\psl}{\mathrm{PSL}(2,\Z)}
\newcommand{\F}{\mathbb{F}}

\title{$C^*$-irreducibility of commensurated subgroups}

\author{Kang Li and Eduardo Scarparo}
\address{Kang Li\\Friedrich-Alexander-Universität Erlangen-Nürnberg \\Germany}
\email{kang.li@fau.de}
\address{Eduardo Scarparo\\ Center for Engineering\\ Federal University of Pelotas\\ Brazil}
\email{eduardo.scarparo@ufpel.edu.br}
\thanks{This project has received funding from the European Research Council (ERC) under the European Union's Horizon 2020 research and innovation programme (grant agreement No. 817597).}

\begin{document}

\begin{abstract}
 Given a commensurated subgroup $\La$ of a group $\G$,  we completely characterize when the inclusion $\La\leq \G$ is $C^*$-irreducible and provide new examples of such inclusions.  In particular, we obtain that $\rm{PSL}(n,\Z)\leq\rm{PGL}(n,\Q)$ is $C^*$-irreducible for any $n\in \N$, and that the inclusion of a $C^*$-simple group into its abstract commensurator is $C^*$-irreducible.

The main ingredient that we use is the fact that the action of a commensurated subgroup $\La\leq\G$ on its Furstenberg boundary $\partial_F\La$ can be extended in a unique way to an action of $\G$ on $\partial_F\La$. Finally, we also investigate the counterpart of this extension result for the universal minimal proximal space of a group.

\end{abstract}

\maketitle
\section{Introduction}

A group $\G$ is said to be \emph{$C^*$-simple} if its reduced $C^*$-algebra $C^*_r(\G)$ is simple. After the breakthrough characterizations of $C^*$-simplicity in \cite{KK17} and \cite{BKKO},  several directions of research applying the new methods in different settings arose. 

One of the recent interesting directions is investigating when inclusions of groups $\La\leq\G$ are \emph{$C^*$-irreducible}, in the sense that every intermediate $C^*$-algebra $B$ in $C^*_r(\La)\subset B \subset  C^*_r(\G)$ is simple. In \cite{R21},  Rørdam started a systematic study of this property and provided a dynamical criterion for an inclusion of groups to be $C^*$-irreducible.  Together with results in \cite{A21}, \cite{U19} and \cite{BO22}, this has provided a complete characterization of $C^*$-irreducibility of an inclusion in the case that $\La$ is a normal subgroup of $\G$.

Recall that a subgroup $\La$ of a group $\G$ is said to be \emph{commensurated} if, for any $g\in\G$, $\La\cap g\La g^{-1}$ has finite index in $\La$.  This is a much more flexible generalization of normal subgroups and finite-index subgroups. For example, for every $n\geq 2$, $\rm{PSL}(n,\Z)$ is an infinite-index commensurated subgroup of the simple group $\rm{PSL}(n,\Q)$. 

In this work, we generalize the above characterization of $C^*$-irreducibility to commensurated subgroups (see Theorem~\ref{thm:cic}).  The main ingredient in our proof is the fact that the action of $\La$ on its Furstenberg boundary $\partial_F\La$ can be uniquely extended to an action of $\G$ on $\partial_F\La$ if $\La$ is a commensurated subgroup in $\G$ (see Theorem \ref{thm:ext}). 

As one of the applications, we show that, if $\G$ is a $C^*$-simple group, then the inclusion of $\G$ in its abstract commensurator $\co(\G)$ is $C^*$-irreducible (see Corollary~\ref{cor:co}). To our best knowledge, this is also the first observation of the fact that, if $\G$ is a $C^*$-simple group, then $\co(\G)$ is $C^*$-simple as well.

Given a subgroup $\La$ of a group $\G$, Ursu introduced in \cite{U19} a universal $\La$-strongly proximal $\G$-boundary $B(\G,\La)$ and showed that, if $\La\unlhd\G$, then $B(\G,\La)=\partial_F\La$. In Section~\ref{sec:rb}, we generalize this fact to commensurated subgroups and also observe that, in general, $B(\G,\La)$ is not extremally disconnected.

Finally, we also show that, given a commensurated subgroup $\La$ of a group $\G$, the action of $\La$ on its universal minimal proximal space $\partial_p\La$ can also be extended in a unique way to an action of $\G$ on $\partial_p\La$ (see Theorem~\ref{thm:prox}), and use this fact for concluding that, for a certain locally finite commensurated subgroup $G$ of Thompson's group $V$, the resulting action of $V$ on $\partial_p G$ is free (see Example \ref{ex:v}).

\section{Preliminaries}
 Given a compact Hausdorff space $X$, we denote by $\Prob(X)$ the space of regular probability measures on $X$. An action of a group $\G$ on $X$ by homeomorphisms is said to be \emph{minimal} if $X$ does not contain any non-trivial closed invariant subset, and to be \emph{topologically free} if, for any $g\in\G\setminus\{e\}$, the set $\{x\in X:gx=x\}$ has empty interior (if $\G$ is countable, then $\G\act X$ is topologically free if and only if the set of points in $X$ which are not fixed by any non-trivial element of $\G$ is dense in $X$).  The action is said to be \emph{proximal} if, given $x,y\in X$, there is a net $(g_i)\subset\G$ such that the nets $(g_ix)$ and $(g_iy)$ converge and $\lim g_ix=\lim g_iy$.  We say that the action is \emph{strongly proximal} if the induced action $\G\act \Prob(X)$ is proximal.
The action is called a \emph{boundary action} (or $X$ is a \emph{$\G$-boundary}) if it is both minimal and strongly proximal.  We denote by $\partial_F\G$ the \emph{Furstenberg boundary} of $\G$, i.e., the universal $\G$-boundary (see \cite[Section III.1]{G76}).  The group $\G$ is $C^*$-simple if and only if $\G\act\partial_F\G$ is free (\cite[Theorem 3.1]{BKKO}).

Given $\G$-boundaries $X$ and $Y$, if there exists $\varphi\colon X\to Y$ a homeomorphism which is $\G$-equivariant ($\G$-\emph{isomorphism}), then it follows from \cite[Lemma II.4.1]{G76} that $\varphi$ is the unique $\G$-isomorphism between $X$ and $Y$.

Let $\La\leq\G$ be a finite-index subgroup.  Then any strongly proximal $\G$-action is also $\La$-strongly proximal (\cite[Lemma II.3.1]{G76}) and any $\G$-boundary is also a $\La$-boundary (\cite[Lemma II.3.2]{G76}).  Furthermore, by \cite[Theorem II.4.4]{G76}, which is stated for the universal minimal proximal space but whose proof also works for the Furstenberg boundary,  the action $\La\act\partial_F\La$ can be extended to $\G\act\partial_F\La$ and $\partial_F\La$ is $\G$-isomorphic to $\partial_F\G$.  In particular, $\partial_F\La$ and $\partial_F\G$ are also $\La$-isomorphic.

Given a group isomorphism $\psi\colon \G_1\to\G_2$,  by universality there is a unique homeomorphism $\tilde{\psi}\colon\partial_F\G_1\to\partial_F\G_2$ such that $\tilde{\psi}(gx)=\psi(g)\tilde{\psi}(x)$ for any $g\in\G_1$ and $x\in\partial_F\G_1$.

Given a group $\G$, let $\Sub(\G)$ be the space of subgroups of $\G$ endowed with the pointwise convergence topology and with the $\G$-action given by conjugation.  Given a subgroup $\La\leq \G$,  a $\La$-\emph{uniformly recurrent subgroup} (URS) is a non-empty closed $\La$-invariant minimal set $\mathcal{U}\subset\Sub(\G)$. Moreover, we say that $\mathcal{U}$ is \emph{amenable} if one (equivalently all) of its elements is amenable. By \cite[Theorem 4.1]{Ken20}, a group $\G$ is $C^*$-simple if and only if it does not admit any non-trivial amenable $\G$-uniformly recurrent subgroup. 

An inclusion of groups $\La\leq \G$ is said to be \emph{$C^*$-irreducible} if every intermediate $C^*$-algebra of $C^*_r(\La)\subset  C^*_r(\G)$ is simple. 

Given $\La\leq\G$ and $g\in\G$, let $g^\La:=\{hgh^{-1}:h\in\La\}$. We say that \emph{$\G$ is icc relatively to $\La$} if, for any $g\in\G\setminus\{e\}$, $|g^\La|<\infty$.  The group $\G$ is said to be \emph{icc} if it is icc relatively to itself.

\section{C*-irreducibility of commensurated subgroups}\label{sec:cic}

Let $\G$ be a group.  Two subgroups $\La_1,\La_2\leq\G$ are said to be \emph{commensurable} if $[\La_1:\La_1\cap\La_2]<\infty$ and $[\La_2:\La_1\cap\La_2]<\infty$.  Notice that this is an equivalence relation. 

A subgroup $\La\leq \G$ is said to be \emph{commensurated} if, for any $g\in\G$,  $\La$ is commensurable with $g\La g^{-1}$. Equivalently, for any $g\in\G$, $[\La:\La\cap g\La g^{-1}]<\infty$.  In this case, we write $\La\leq_c\G$.  In the literature, this notion is also referred to by saying that $\La$ is an \emph{almost normal subgroup} of $\G$ or that $(\G,\La)$ is a \emph{Hecke pair}.

The following result generalizes \cite[Theorem II.4.4]{G76} and \cite[Lemma 20]{O14}.

\begin{theorem}\label{thm:ext}
Let $\La\leq_c\G$. Then $\La\act\partial_F\La$ extends in a unique way to an action of $\G$ on $\partial_F\La$.

\end{theorem}
\begin{proof}
Given $g\in\G$, let $\varphi_g\colon\partial_F\La\to\partial_F(\La\cap g\La g^{-1})$ be the $(\La\cap g\La g^{-1})$-isomorphism. Also let $\psi_g\colon\partial_F(\La\cap g^{-1}\La g)\to\partial_F(\La\cap g\La g^{-1})$ be the homeomorphism such that for all $h\in \La\cap g^{-1}\La g$ and $x\in\partial_F( \La\cap g^{-1}\La g)$ we have $\psi_g(hx)=ghg^{-1}\psi_g(x)$. Let $T_g:=(\varphi_g)^{-1}\psi_g\varphi_{g^{-1}}\colon \partial_F\La\to\partial_F\La$.  We claim that $g\mapsto T_g$ is a $\G$-action which extends $\La\act\partial_F\La$. 

Given $h\in \La\cap g^{-1}\La g$ and $x\in\partial_F\La$, one can readily check that $T_g(hx)=ghg^{-1}T_g(x)$.

Given $g,h\in\G$,  we have that $[\La:\La\cap h^{-1}\La h\cap (gh)^{-1}\La(gh)]<\infty$. Furthermore, given $k\in\La\cap h^{-1}\La h\cap (gh)^{-1}\La(gh)$ and $x\in\partial_F\La$, we have $T_{gh}(kx)=(gh)k(gh)^{-1}T_{gh}(x)$. On the other hand, $T_gT_h(kx)=(gh)k(gh)^{-1}T_gT_h(x)$.  In particular, $(T_gT_h)^{-1}T_{gh}$ is a $(\La\cap h^{-1}\La h\cap (gh)^{-1}\La(gh))$-automorphism,  hence $T_{gh}=T_gT_h$.

Finally, given $g\in\La$, we have that $x\mapsto g^{-1}T_g(x)$ is a $(\La\cap g^{-1}\La g)$-automorphism, so that $g^{-1}T_g=\Id_{\partial_F\La}$.
\end{proof}

\begin{remark}
The existence part of Theorem \ref{thm:ext} was shown by Dai and Glasner in \cite[Theorem 6.1]{DG19} using a different method and assuming that $\G$ is countable. 

\end{remark}

Given a subset $S$ of a group $\G$, let $C_\G(S)$ be the \emph{centralizer} of $S$ in $\G$. 
In the next result, we follow the argument of \cite[Lemma 5.3]{BKKO}.
\begin{lemma}\label{lem:fix}
Let $\La\leq_c\G$ and consider $\G\act\partial_F\La$. Given $s\in\G$, if $s\in \mathrm{C}_\G(\La\cap s^{-1}\La s)$, then $\Fix(s)=\partial_F\La$.  Conversely, if $\La\act\partial_F\La$ is free and $\Fix(s)\neq\emptyset$, then $s\in \mathrm{C}_\G(\La\cap s^{-1}\La s)$.
\end{lemma}
\begin{proof}
 If $s\in\mathrm{C}_\G(\La\cap s^{-1}\La s)$, then, given $h\in\La\cap s^{-1}\La s$ and $x\in \partial_F\La$, we have $s(hx)=hs(x)$. Since $[\La:\La\cap s^{-1}\La s]<\infty$, we conclude that $s$ acts trivially on $\partial_F\La$.

Suppose now that $\La\act\partial_F\La$ is free and $\Fix(s)\neq\emptyset$.  Given $t\in A:=\{t\in\La\cap s^{-1}\La s:t\Fix(s)\cap\Fix(s)\neq\emptyset\}$, we have that the action of $sts^{-1}$ and $t$ coincide on $\Fix(s)\cap t^{-1}\Fix(s)$. Since $sts^{-1},t\in\La$ and $\La\act\partial_F\La$ is free, we obtain that $t=sts^{-1}$. Since,  by \cite[Lemma 5.1]{BKKO}, $A$ generates $\La\cap s^{-1}\La s$, we conclude that $s\in\mathrm{C}_\G(\La\cap s^{-1}\La s)$.
\end{proof}

The proof of the following result is an adaptation of the argument in \cite[Remark 4.2]{Ken20} and its hypothesis is the same as in \cite[Theorem 5.3.(ii)]{R21}.
\begin{proposition}\label{prop:urs}
Let $\La\leq\G$.  Suppose that there exists a $\G$-boundary $X$ such that, for any $\mu\in\Prob(X)$, there exists a net $(g_i)\subset\La$ such that $g_i\mu$ converges to $\delta_x$, for some $x\in X$, on which $\G$ acts freely. Then $\G$ does not admit any non-trivial amenable $\La$-URS.
\end{proposition}
\begin{proof}
Suppose $\mathcal{U}$ is a non-trivial amenable $\La$-URS, and take $K\in \mathcal{U}$.  Since $K$ is amenable, there exists $\mu\in\Prob(X)$ fixed by $K$. Let $(g_i)\subset \La$ be a net such that $g_i\mu\to\delta_x$, for some $x\in X$, on which $\G$ acts freely. By taking a subnet, we may assume that $g_iKg_i^{-1}\to L\in\Sub(\G)$. Take $g\in L\setminus\{e\}$ and $(k_i)\subset K$ such that $g_ik_ig_i^{-1}=g$ for $i$ sufficiently big. Then $$\delta_x=\lim g_i\mu =\lim g_ik_i\mu=\lim g_ik_ig_i^{-1}g_i\mu=g\delta_x,$$
contradicting the fact that $\G$ acts freely on $x$.
\end{proof}

The following result generalizes \cite[Theorems 1.3 and 1.9]{U19} and \cite[Theorem 6.4]{BO22}, as well as the claim about finite-index subgroups in \cite[Theorem 5.3]{R21}.
\begin{theorem}\label{thm:cic}
Let $\La\leq_c \G$. The following conditions are equivalent:

\begin{enumerate}
\item $\La\leq\G$ is $C^*$-irreducible;
\item $\La$ is $C^*$-simple and $\G$ is icc relatively to $\La$;
\item $\La$ is $C^*$-simple and, for any $s\in\G\setminus\{e\}$, we have that $s\notin {C}_\G(\La\cap s^{-1}\La s)$;
\item $\G\act \partial_F\La$ is free;

\item There is no non-trivial amenable $\La$-URS of $\G$;
\item $\La$ is $C^*$-simple and $\G\act\partial_F\La$ is faithful.
\end{enumerate}

\end{theorem}
\begin{proof}
(1)$\implies$(2) follows from \cite[Remark 3.8 and Proposition 5.1]{R21}.

(2)$\implies$(3). Suppose that there is $s\in\G\setminus\{e\}$ such that $s\in {C}_\G(\La\cap s^{-1}\La s)$. Take $g_1,\dots,g_n\in\La$ left coset representatives for $\frac{\La}{\La\cap s^{-1}\La s}$.  Then $$s^\La=\{g_iksk^{-1}g_i^{-1}:1\leq i\leq n,k\in \La\cap s^{-1}\La s\}=\{g_isg_i^{-1}:1\leq i\leq n\}$$ is finite.

(3)$\implies$(4) follows from Lemma \ref{lem:fix}.

(4)$\implies$(1) follows from \cite[Theorem 5.3]{R21}.

(5)$\implies$(2). If $\La$ is not $C^*$-simple, then it contains a non-trivial amenable $\La$-uniformly recurrent subgroup. If $\G$ is not icc relatively to $\La$, there exists $s\in\G\setminus\{e\}$ such that $s^\La$ is finite. Hence the $\La$-orbit of $\langle s\rangle $ is a finite non-trivial amenable $\La$-uniformly recurrent subgroup.

(4)$\implies$(5) follows from Proposition \ref{prop:urs}.

(3)$\iff$(6) follows from Lemma \ref{lem:fix}.
\end{proof}

\begin{remark}
In \cite[Theorem 5.3]{R21}, Rørdam showed that an inclusion $\La\leq\G$ satisfying the hypothesis of Proposition \ref{prop:urs} is $C^*$-irreducible, and asked whether the converse holds. We do not know whether the converse of Proposition \ref{prop:urs} holds and whether the absence of non-trivial amenable $\La$-URS of $\G$ is equivalent to $\La\leq\G$ being $C^*$-irreducible in general.
\end{remark}

\begin{corollary}\label{cor:psl}
Given $n\in\N$, the inclusion $$\rm{PSL}(n,\Z)\leq\mathrm{PGL}(n,\Q)$$ is $C^*$-irreducible.
\end{corollary}
\begin{proof}
It was shown in \cite{BCdlH94} that $\mathrm{PSL}(n,\Z)$ is $C^*$-simple.  

Let $U(n,\Z)$ be the group of units of the ring $M_n(\Z)$. By \cite[Corollary V.5.3]{K90}, $U(n,\Z)\leq_c\rm{GL}(n,\Q)$. Since $[U(n,\Z):\rm{SL}(n,\Z)]=2$,  we conclude that $\rm{SL}(n,\Z)\leq_c\rm{GL}(n,\Q)$ as well. Since taking quotients preserves being commensurated, it follows that $\rm{PSL}(n,\Z)\leq_c\mathrm{PGL}(n,\Q)$. 

Let $(e_{ij})_{1\leq i,j\leq n}\in M_n(\Z)$ be the matrix units and fix $[a]\in\rm{PGL}(n,\Q)\setminus\{\rm{[Id]}\}$. By taking conjugates of $[a]$ by elements of the form $[\rm{Id}+m\cdot e_{ij}]\in\rm{PSL}(n,\Z)$,  $m\in\Z$, $1\leq i\neq j\leq n$, it is easy to see that $[a]^{\rm{PSL}(n,\Z)}$ is infinite, so that $\mathrm{PGL}(n,\Q)$ is icc relatively to $\rm{PSL}(n,\Z)$.

The conclusion then follows from Theorem \ref{thm:cic}.
\end{proof}
\iffalse
\begin{example}
If $p$ is a prime number, then by Corollary~\ref{cor:psl} the inclusion $$\rm{PSL}(n,\Z)\leq\mathrm{PSL}(n,\Z[1/p])$$ is $C^*$-irreducible for any $n\in\N$. Note that every
normal subgroup of the group $\mathrm{PSL}(n,\Z[1/p])$ is finite or finite-index, but
$\rm{PSL}(n,\Z)$ is an infinite and infinite-index commensurated subgroup.
\end{example}
\fi

\begin{remark}
Let us sketch a different proof of Corollary \ref{cor:psl} which gives the stronger statement that $\rm{PSL}(n,\Z)\leq\rm{PGL}(n,\R)$ is $C^*$-irreducible, where $\rm{PGL}(n,\R)$ is seen as a discrete group.  

Clearly, it suffices to show that, for any countable group $\G$ such that $\rm{PSL}(n,\Z)\leq\G\leq\rm{PGL}(n,\R)$, the inclusion $\rm{PSL}(n,\Z)\leq\G$ is $C^*$-irreducible. By the argument in \cite[Example 3.4.3]{Bry17}, the action of $\rm{PGL}(n,\R)$ on the projective space $P^{n-1}(\R)$ is topologically free.  Since $\rm{PSL}(n,\Z)\act P^{n-1}(\R)$ is a boundary action,  the result follows from \cite[Theorem 5.3]{R21}.
\end{remark}

\begin{corollary}\label{cor:fi}
Let $\La$ be a finite-index subgroup of a group $\G$.  If $\G$ is $C^*$-simple,  then $\La\leq\G$ is $C^*$-irreducible.  Conversely, if $\La$ is $C^*$-simple, then $\G$ is icc if and only if $\La\leq\G$ is $C^*$-irreducible.

\end{corollary}
\begin{proof}
If $\G$ is $C^*$-simple, then $\G\act\partial_F\G$ is free. Since $\partial_F\G$ is $\G$-isomorphic to $\partial_F\La$, it follows that $\La\leq\G$ is $C^*$-irreducible.

If $\G$ is icc, then, since $[\G:\La]<\infty$, it is also icc relatively to $\La$, hence $\La\leq\G $ is $C^*$-irreducible by Theorem~\ref{thm:cic}.  The last implication is immediate.

\end{proof}

\begin{example}
The inclusion given by the Sanov subgroup $\mathbb{F}_2\leq \rm{PSL}(2,\Z)$ is finite-index, hence it is $C^*$-irreducible by Corollary~\ref{cor:fi}.
\end{example}

\subsection*{Free groups}
Fix $m,n\in\N$ such that $2\leq m< n$ and consider the free groups $\mathbb{F}_m=\langle a_1,\dots,a_m\rangle\leq\langle a_1,\dots,a_n\rangle=\mathbb{F}_n$.  In \cite[Example 5.4]{R21}, Rørdam observed that $\mathbb{F}_m\leq \mathbb{F}_n$ is $C^*$-irreducible. Notice that $\mathbb{F}_m$ is far from being commensurated in $\mathbb{F}_n$. In fact, given $g\in \mathbb{F}_n\setminus \mathbb{F}_m$, we have that $\mathbb{F}_m\cap g \mathbb{F}_m g^{-1}=\{e\}$ (i.e., $\mathbb{F}_m$ is \emph{malnormal} in $\mathbb{F}_n$). In particular, this example is not covered by Theorems \ref{thm:ext} and \ref{thm:cic}. Nonetheless,  there does exist an extension to $\mathbb{F}_n$ of the action $\mathbb{F}_m\act\partial_F \mathbb{F}_m$, but it is far from being unique, since the generators $a_{m+1},\dots,a_n$ can be mapped into any homeomorphisms on $\partial_F \mathbb{F}_m$. 

Furthermore, we claim that $\mathbb{F}_m\leq \mathbb{F}_n$ satisfies condition (5) in Theorem \ref{thm:cic}.  We will prove this by using Proposition \ref{prop:urs}.

 Let 
$$\partial \mathbb{\F}_n:=\{(x_i)\in\prod_\N\{a_1,a_1^{-1},\dots,a_n,a_n^{-1}\}:\forall i\in\N, x_{i+1}\neq x_i^{-1}\}$$
be the Gromov boundary of $\mathbb{F}_n$, and consider the action of $\F_n$ on $\partial\F_n$ by left multiplication.  Fix $\mu\in\Prob(\partial \F_n)$ and we will show that there is $w\in \partial\F_n$ on which $\F_n$ acts freely and such that $\delta_w\in\overline{\F_m\mu}$.

Let $z_+:=(a_1)_{i\in\N}\in\partial\F_n$ and $z_-:=(a_1^{-1})_{i\in\N}\in\partial\F_n$. Notice that, for all $y\in\partial \F_n\setminus\{z_-\},$ we have that, as $k\to+\infty$, $a_1^ky\to z_+$. Furthermore, $a_1$ fixes $z_-$.

It follows from the dominated convergence theorem that $$a_1^k\mu\to\mu(\{z_-\})\delta_{z_-}+(1-\mu(\{z_-\})\delta_{z_+}, $$
as $k\to+\infty$. In particular,  $\nu:=\mu(\{z_-\})\delta_{z_-}+(1-\mu(\{z_-\})\delta_{z_+}\in\overline{\F_n\mu}$.

Let $w:=a_1a_2^1a_1a_2^2a_1a_2^3\dots a_1a_2^la_1a_2^{l+1}\dots\in\partial\F_n$. Since $w$ is not eventually periodic, we have that $\F_n$ acts freely on $w$. Given $k\in\N$, let $g_k:=w_1\dots w_ka_2\in\F_m$.  We have that $g_ kz_{\pm}=w_1\dots w_ka_2z_{\pm}\to w$, as $k\to+\infty$. Therefore, $\delta_w\in \overline{\F_m\nu}\subset\overline{\F_m\mu}$, thus showing the claim.

\subsection*{Abstract commensurator} Let $\G$ be a group and $\Omega$ be the set of isomorphisms between finite-index subgroups of $\G$. Given $\alpha, \beta\in\Omega$,  we say that $\alpha\sim\beta$ if there exists a finite-index subgroup $H\leq \dom(\alpha)\cap\dom(\beta)$ such that $\alpha|_H=\beta|_H$.  Recall that the \emph{abstract commensurator} of $\G$, denoted by $\co(\G)$, is the group whose underlying set is $\Omega/{\sim}$, with product given by composition (defined up to finite-index subgroup).  

Let $\La$ be a commensurated subgroup of $\G$. Given $g\in\G$, let 

\begin{align*}
\beta_g\colon \La\cap g^{-1}\La g&\to \La\cap g\La g^{-1}\\
h&\mapsto ghg^{-1}
\end{align*}
and $j_\La^\G\colon\G\to\co(\La)$ be the homomorphism given by $j^\G_\La(g):=[\beta_g]$.  In order to ease the notation, we will sometimes denote $j_\La^\G$ simply by $j$, and it will always be clear from the context what are the involved groups.  
Let us now collect a few elementary facts about $j$.
\begin{lemma}
Let $\G$ be a group. Then $j_\G^\G(\G)\leq_c\co(\G)$. 
\end{lemma}
\begin{proof}
Fix $[\alpha]\in\co(\G)$. Given $g\in\dom(\alpha)$, we have that $[\alpha] j(g)[\alpha]^{-1}=j(\alpha(g))$.  In particular,  $j(\G)\cap[\alpha] j(\G)[\alpha]^{-1}\supset j(\Ima( \alpha)).$ Since $[\G:\Ima(\alpha)]<\infty$, we conclude that $[j(\G):j(\G)\cap[\alpha] j(\G)[\alpha]^{-1}]<\infty.$
\end{proof}

\begin{lemma}\label{lem:ker} Let $\La\leq_c\G$. Then $\ker j^\G_\La=\{g\in\G:|g^\La|<\infty\}.$
\end{lemma}
\begin{proof}
Given $g\in\ker j$, there exists a finite-index subgroup $H\leq\La\cap g^{-1}\La g$ such that, for all $h\in H$, $ghg^{-1}=h$, which implies that $|g^\La|<\infty$. Conversely, if $|g^\La|<\infty$, then $H:=\{k\in\La:kg=gk\}$ is a finite-index subgroup of $\La$ and $g\in\ker j$.
\end{proof}
As a consequence of Lemma \ref{lem:ker}, if $\G$ is an icc group,  then $j\colon \G\to\co(\G)$ is injective (\cite[Lemma 3.8.(i)]{K11}). The next result is known (\cite[Lemma 3.8.(iii)]{K11}).  For the convenience of the reader, we provide the proof here. 
\begin{lemma}\label{lem:icc}
If $\G$ is an icc group, then $\co(\G)$ is icc relatively to $\G$.
\end{lemma}
\begin{proof}
Given $[\alpha]\in\co(\G)$ and $g\in\dom(\alpha)$, we have $$j(g)[\alpha]j(g^{-1})=j(g\alpha(g^{-1}))[\alpha].$$ 
If $[\alpha]\neq e$, then $H:=\{g\in\dom(\alpha):g=\alpha(g)\}$ has infinite-index in $\dom(\alpha)$. Given $g_1,g_2\in\dom(\alpha)$ such that $g_1H\neq g_2 H$, one can readily check that $g_1\alpha(g_1)^{-1}\neq g_2\alpha(g_2)^{-1}$.  From this, it follows immediately that $[\alpha]^\G$ is infinite.
\end{proof}

In \cite[Corollary 6.6]{BO22}, Bédos and Omland showed that if $\G$ is a $C^*$-simple group, then $\G\leq\mathrm{Aut}(\G)$ is $C^*$-irreducible.  The same conclusion holds when we consider the abstract commensurator:
\begin{corollary}\label{cor:co}
Given a $C^*$-simple group $\G$, we have that $\G\leq\co(\G)$ is $C^*$-irreducible.
\end{corollary}
\begin{proof}
Recall that any $C^*$-simple group is icc (this follows, e.g., from Theorem \ref{thm:cic}).  The result is then a consequence of Theorem \ref{thm:cic} and Lemma \ref{lem:icc}.
\end{proof}
\begin{remark}
Corollary \ref{cor:co} generalizes the fact proven in \cite[Corollary 4.4]{LBMB18} that, if Thompson's group $F$ is $C^*$-simple, then $\co(F)$ is $C^*$-simple.
\end{remark}
\begin{remark}
Let $\mathbb{F}_n$ be a non-abelian free group of finite rank. Then Corollary \ref{cor:co} implies that $\co(\mathbb{F}_n)$ is $C^*$-simple.  In particular, it does not admit any non-trivial amenable normal subgroup. It is an open problem whether $\co(\mathbb{F}_n)$ is a simple group (\cite[Problem 7.2]{CM18}). 
\end{remark}

\section{Relative boundaries}\label{sec:rb}

Given groups $\La\leq\G$, Ursu introduced in \cite[Proposition 4.1]{U19} a $\La$-strongly proximal $\G$-boundary $B(\G,\La)$ which is universal with these properties.

Consider $\G:=\psl$ and the boundary action $\G\act\R\cup\{\infty\}$.  The stabilizer $\G_\infty$ of $\infty$ is isomorphic to $\Z$ and consists of the translations $g_n(x):=x+n$, $n\in\Z$, $x\in\R$. 

\begin{proposition}\label{prop:ed}
The action of $\G=\rm{PSL}(2,\Z)$ on $ B(\G,\G_\infty)$ is topologically free but non-free. In particular,  $B(\G,\G_\infty)$ is not extremally disconnected.
\end{proposition}
\begin{proof}
For any $x\in\R\cup\{\infty\}$, we have $g_n(x)\to \infty$ as $n\to+\infty$.  As a consequence of the dominated convergence theorem,  it follows easily that $\G_\infty\act\R\cup\{\infty\}$ is strongly proximal. Hence, there is a $\G$-equivariant map $B(\G,\G_\infty)\to\R\cup\{\infty\}$.  Since $\G_\infty\act B(\G,\G_\infty)$ is strongly proximal, it follows from amenability of $\G_{\infty}$ that $\G_\infty$ fixes some point in $B(\G,\G_\infty)$. In particular, $\G\act B(\G,\G_\infty)$ is not free. On the other hand, since $\G\act\R\cup\{\infty\}$ is topologically free, it follows from \cite[Lemma 3.2]{BKKO} that $\G\act B(\G,\G_\infty)$ is topologically free.  As a consequence of \cite[Theorem 3.1]{F71}, $B(\G,\G_\infty)$ is not extremally disconnected.
\end{proof}
\begin{remark}
Let $\G$ be a group. One of the key properties in the applications of $\partial_F\G$ to $C^*$-simplicity of $\G$ is the fact that $C(\partial_F\G)$ is injective, shown in \cite[Theorem 3.12]{KK17}. Proposition \ref{prop:ed} implies that $C(B(\G,\La))$ is not injective, in general.  We believe that this is an evidence that $B(\G,\La)$ is not likely to play the same role of the Furstenberg boundary in $C^*$-algebraic applications.
\end{remark}
Our next aim is to show that, given $\La\leq_c\G$, it holds that $B(\G,\La)=\partial_F\La$. We start with a result which we believe has its own interest.

\begin{theorem}
Let $\La\leq_c \G$ and $\G\act X$ a minimal action on a compact space such that $\La\act X$ is proximal. Then $\La\act X$ is minimal as well. 
\end{theorem}
\begin{proof}
 Let $M\subset X$ be a closed non-empty $\La$-invariant set. For any $g\in\G$, we have that $gM$ is $g\La g^{-1}$-invariant.

Fix $g_1,\dots,g_n\in\G$. We have that $H:=\La\cap g_1\La g_1^{-1}\cap\dots\cap g_n\La g_n^{-1}$ has finite index in $\La$. In particular,  $H\act X$ is proximal and admits a unique minimal component $K$. Since each $g_i M$ is $g_i\La g_i^{-1}$-invariant, we conclude that $K\subset\bigcap_{i=1}^n g_iM$. 

By compactness of $X$, we obtain that $L:=\bigcap_{g\in\G} gM\neq\emptyset$. Since $L$ is $\G$-invariant, we have $X=L\subset M$. 
\end{proof}

The following is an immediate consequence of the previous theorem: 
\begin{corollary}
Let $\La\leq_c \G$.  If $X$ is a $\G$-boundary which is also $\La$-strongly proximal, then $X$ is a $\La$-boundary.
\end{corollary}

By arguing as in \cite[Corollary 4.3]{U19}, we conclude the following:
\begin{corollary}\label{cor:id}

If $\La\leq_c\G$, then $B(\G,\La)=\partial_F\La$.
\end{corollary}

\section{Commensurated subgroups and proximal actions}

Given a group $\G$, there exists a universal minimal proximal $\G$-space $\partial_p \G$ (\cite[Theorem II.4.2]{G76}). It was shown in \cite[Proposition 2.12]{FTF19} and \cite[Theorem 1.5]{GTWZ21} that a countable group $\G$ is icc if and only if $\G\act\partial_p\G$ is faithful, if and only if $\G\act\partial_p\G$ is free.

One can easily check that the statements of Theorem \ref{thm:ext} and Lemma \ref{lem:fix} hold with $\partial_p\La$ instead of $\partial_F\La$, with the exact same proofs (in particular,  \cite[Lemma 5.1]{BKKO}, which is needed in the proof of Lemma \ref{lem:fix},  uses only proximality).  Thus, we obtain:

\begin{theorem}\label{thm:prox}
Let $\La\leq_c\G$. Then $\La\act\partial_p\La$ extends in a unique way to an action of $\G$ on $\partial_p\La$.   Furthermore,  given $s\in\G$, if $s\in \mathrm{C}_\G(\La\cap s^{-1}\La s)$, then $\Fix(s)=\partial_p\La$.  Conversely, if $\La\act\partial_p\La$ is free and $\Fix(s)\neq\emptyset$, then $s\in \mathrm{C}_\G(\La\cap s^{-1}\La s)$.
\end{theorem}
As a consequence, we obtain the following: 
\begin{theorem} \label{thm:ump}

Let $\La\leq_c\G$ and suppose that $\La\act\partial_p\La$ is free.  The following conditions are equivalent:

\begin{enumerate}
\item  $\G$ is icc relatively with $\La$;
\item For any $s\in\G\setminus\{e\}$,  we have that $s\notin {C}_\G(\La\cap s^{-1}\La s)$;
\item  $\G\act\partial_p\La$ is free;
\item  $\G\act\partial_p\La$ is faithful.

\end{enumerate}
\end{theorem}
\begin{proof}
The implications (1)$\implies$(2)$\implies$(3)$\implies$(4) are proven as in Theorem \ref{thm:cic}.

(4)$\implies$(1).  Suppose that there is $g\in\G\setminus\{e\}$ such that $|g^\La|<\infty$.  Then $H:=\{h\in\La:gh=hg\}$ is a finite-index subgroup of $\La$, hence $H\act\partial_p\La$ is also minimal and proximal. Since the homeomorphism on $\partial_p\La$ given by $g$ is $H$-equivariant, we conclude that $g$ acts trivially on $\partial_p\La$.
\end{proof}

\begin{remark} Given a group $\G$, let $L(\G)$ be its \emph{group von Neumann algebra}.  Given $\La\leq\G$,  it follows from \cite[Proposition 5.1]{R21} and \cite[Corollary 4.3]{BO22} that $\G$ is icc relatively to $\La$ if and only if any intermediate von Neumann algebra of $L(\La)\subset L(\G)$ is a factor, if and only if any intermediate $C^*$-algebra of $C^*_r(\La)\subset C^*_r(\G)$ is prime.
\end{remark}

Let us now apply Theorem \ref{thm:ump} to a certain locally finite commensurated subgroup of Thompson's group $V$. 

\begin{example}\label{ex:v}Let $X:=\{0,1\}$ and, given $n\geq 0$, let $X^n$ be the set of words in $X$ of length $n$.  Given $w\in X^n$,  let $\cC(w):=\{(s_n)\in X^\N:s_{[1,n]}=w\}$.  Recall that Thompson's group $V$ is the group of homeomorphisms on $X^\N$ consisting of elements $g$ for which there exist two partitions $\{\cC(w_1),\dots\cC(w_m)\}$ and $\{\cC(z_1),\dots,\cC(z_m)\}$ of $\{0,1\}^\N$ such that $g(w_is)=z_is$ for every $1\leq i \leq m$ and $s\in X^\N$. 

Let us define inductively groups $G_n$ acting by permutations on $X^n$. 

Let $G_1:=\Z_2$ acting non-trivially on $X$ and, for $n\in\N$, 

$$G_{n+1}:=\left(\bigoplus_{w\in X^n}\Z_2\right)\rtimes G_n,$$
where the action of $G_{n+1}$ on $X^{n+1}$ is defined as follows: given $v\in X^n$, $x\in X$, $\sigma\in G_n$ and $f\in\bigoplus_{X^n}\Z_2$, 

 $$(f,\sigma)(vx):=\sigma(v)f_{\sigma(v)}(x).$$

Let $G:=\lim_{n\in \N} G_n$.  Then $G$ acts faithfully on $X^\N$ and, as observed in \cite[Proposition 7.11]{LB17}, $G\leq_c V$.  

We claim that $V$ is icc relatively with $G$.  Given $u\in X^n$, let the \emph{rigid stabilizer} of $u$, denoted by $\mathrm{rist}_G(u)$,  be the subgroup of $G$ consisting of the elements which, for every $v\in X^n\setminus\{u\}$, act as the identity on $\cC(v)$.  Given $g\in G$, there is $\tilde{g}\in \mathrm{rist}_G(u)$ such that $\tilde{g}(us)=ug(s)$ for any $s\in X^\N$.  Clearly, the map $g\mapsto \tilde{g}$ is an isomorphism from $G$ to $\mathrm{rist}_G(u)$. Fix $h\in V\setminus\{e\}$ and take $w\in X^n$ and $z\in X^m$ such that $w\neq z$, $n\geq m$ and $h(ws)=zs$ for any $s\in X^\N$.  Furthermore, take $v\in X^{n-m}$ such that $zv\neq w$. Given $s\in X^\N$,  we have that 
\begin{equation}\label{eq:c}
\{\tilde{g}h\tilde{g}^{-1}(wvs):\tilde{g}\in \mathrm{rist}_G(zv)\}=\{zvg(s):g\in G\}.
\end{equation}
Since $G\act X^\N$ is faithful,  it follows from \eqref{eq:c} that $|h^G|=\infty$, thus proving the claim.

From \cite[Theorem 1.5]{GTWZ21},  we obtain that $G\act\partial_pG$ is free and from Theorem \ref{thm:ump},  we conclude that $V\act\partial_pG$ is free.
\end{example}

\begin{remark}
In \cite[Theorem 1.5]{LBMB18}, Le Boudec and Matte Bon showed that Thompson's group $V$ is $C^*$-simple, hence $V\act\partial_F V$ is free. However, their proof is done by showing that $V$ does not admit non-trivial amenable URS, not by exhibitting a concrete topologically free $V$-boundary. It seems as an interesting problem to determine whether $V\act\partial_pG$ is strongly proximal, thus providing an alternative proof of $C^*$-simplicity of $V$.
\end{remark}
\begin{remark} In \cite[Theorem 1.4]{BKKO}, it was shown that the class of $C^*$-simple groups is closed by taking normal subgroups. Obviously,  this class is not closed by taking commensurated subgroups, since any finite subgroup is commensurated.  Moreover, Example \ref{ex:v} shows that,  given $\La\leq_c\G$ such that $\G$ is icc relatively to $\La$,  $C^*$-simplicity of $\G$ does not pass to $\La$ in general.

\end{remark}
\bibliography{bibliografia}

% \bib, bibdiv, biblist are defined by the amsrefs package.
\begin{bibdiv}
\begin{biblist}

\bib{A21}{article}{
      author={Amrutam, Tattwamasi},
       title={On intermediate subalgebras of {$C^*$}-simple group actions},
        date={2021},
     journal={Int. Math. Res. Not. IMRN},
      number={21},
       pages={16193\ndash 16204},
}

\bib{BCdlH94}{article}{
      author={Bekka, M.},
      author={Cowling, M.},
      author={de~la Harpe, P.},
       title={Simplicity of the reduced {$C^*$}-algebra of {${\rm PSL}(n,{\bf
  Z})$}},
        date={1994},
     journal={Internat. Math. Res. Notices},
      number={7},
       pages={285ff., approx. 7 pp.},
}

\bib{BKKO}{article}{
      author={Breuillard, Emmanuel},
      author={Kalantar, Mehrdad},
      author={Kennedy, Matthew},
      author={Ozawa, Narutaka},
       title={{$C^*$}-simplicity and the unique trace property for discrete
  groups},
        date={2017},
     journal={Publ. Math. Inst. Hautes \'Etudes Sci.},
      volume={126},
       pages={35\ndash 71},
}

\bib{BO22}{article}{
      author={B{\'e}dos, Erik},
      author={Omland, Tron},
       title={{{\(C^\ast \)}}-irreducibility for reduced twisted group
  {{\(C^\ast \)}}-algebras},
        date={2023},
        ISSN={0022-1236},
     journal={J. Funct. Anal.},
      volume={284},
      number={5},
       pages={31},
        note={Id/No 109795},
}

\bib{Bry17}{thesis}{
      author={Bryder, Rasmus~Sylvester},
       title={Boundaries, injective envelopes, and reduced crossed products},
        type={Ph.D. Thesis},
        date={2017},
}

\bib{CM18}{incollection}{
      author={Caprace, Pierre-Emmanuel},
      author={Monod, Nicolas},
       title={Future directions in locally compact groups: a tentative problem
  list},
        date={2018},
   booktitle={New directions in locally compact groups},
      series={London Math. Soc. Lecture Note Ser.},
      volume={447},
   publisher={Cambridge Univ. Press, Cambridge},
       pages={343\ndash 355},
}

\bib{DG19}{article}{
      author={Dai, Xiongping},
      author={Glasner, Eli},
       title={On universal minimal proximal flows of topological groups},
        date={2019},
     journal={Proc. Amer. Math. Soc.},
      volume={147},
      number={3},
       pages={1149\ndash 1164},
}

\bib{F71}{inproceedings}{
      author={Frol\'{\i}k, Z.},
       title={Maps of extremally disconnected spaces, theory of types, and
  applications},
        date={1971},
   booktitle={General {T}opology and its {R}elations to {M}odern {A}nalysis and
  {A}lgebra, {III} ({P}roc. {C}onf., {K}anpur, 1968)},
   publisher={Academia, Prague},
       pages={131\ndash 142},
}

\bib{FTF19}{article}{
      author={Frisch, Joshua},
      author={Tamuz, Omer},
      author={Vahidi~Ferdowsi, Pooya},
       title={Strong amenability and the infinite conjugacy class property},
        date={2019},
     journal={Invent. Math.},
      volume={218},
      number={3},
       pages={833\ndash 851},
}

\bib{G76}{book}{
      author={Glasner, Shmuel},
       title={Proximal flows},
      series={Lecture Notes in Mathematics, Vol. 517},
   publisher={Springer-Verlag, Berlin-New York},
        date={1976},
}

\bib{GTWZ21}{article}{
      author={Glasner, Eli},
      author={Tsankov, Todor},
      author={Weiss, Benjamin},
      author={Zucker, Andy},
       title={Bernoulli disjointness},
        date={2021},
     journal={Duke Math. J.},
      volume={170},
      number={4},
       pages={615\ndash 651},
}

\bib{Ken20}{article}{
      author={Kennedy, Matthew},
       title={An intrinsic characterization of {$C^*$}-simplicity},
        date={2020},
     journal={Ann. Sci. \'{E}c. Norm. Sup\'{e}r. (4)},
      volume={53},
      number={5},
       pages={1105\ndash 1119},
}

\bib{K11}{article}{
      author={Kida, Yoshikata},
       title={Rigidity of amalgamated free products in measure equivalence},
        date={2011},
     journal={J. Topol.},
      volume={4},
      number={3},
       pages={687\ndash 735},
}

\bib{KK17}{article}{
      author={Kalantar, Mehrdad},
      author={Kennedy, Matthew},
       title={Boundaries of reduced {$C^*$}-algebras of discrete groups},
        date={2017},
     journal={J. Reine Angew. Math.},
      volume={727},
       pages={247\ndash 267},
}

\bib{K90}{article}{
      author={Krieg, Aloys},
       title={Hecke algebras},
        date={1990},
     journal={Mem. Amer. Math. Soc.},
      volume={87},
      number={435},
       pages={x+158},
}

\bib{LB17}{article}{
      author={Le~Boudec, Adrien},
       title={Compact presentability of tree almost automorphism groups},
        date={2017},
     journal={Ann. Inst. Fourier (Grenoble)},
      volume={67},
      number={1},
       pages={329\ndash 365},
}

\bib{LBMB18}{article}{
      author={Le~Boudec, Adrien},
      author={Matte~Bon, Nicol\'{a}s},
       title={Subgroup dynamics and {$C^*$}-simplicity of groups of
  homeomorphisms},
        date={2018},
     journal={Ann. Sci. \'{E}c. Norm. Sup\'{e}r. (4)},
      volume={51},
      number={3},
       pages={557\ndash 602},
}

\bib{O14}{unpublished}{
      author={Ozawa, Narutaka},
       title={Lecture on the {F}urstenberg boundary and
  {{\(C^{\ast}\)}}-simplicity},
        date={2014},
        note={Available at
  \url{http://www.kurims.kyoto-u.ac.jp/~narutaka/notes/yokou2014.pdf}},
}

\bib{R21}{article}{
      author={R{\o}rdam, Mikael},
       title={Irreducible inclusions of simple {{\(C^{\ast}\)}}-algebras},
        date={2021},
     journal={arXiv preprint arXiv:2105.11899},
}

\bib{U19}{article}{
      author={Ursu, Dan},
       title={Relative {$\rm C^*$}-simplicity and characterizations for normal
  subgroups},
        date={2022},
     journal={J. Operator Theory},
      volume={87},
      number={2},
       pages={471\ndash 486},
}

\end{biblist}
\end{bibdiv}
\end{document}